\documentclass[12pt]{amsart}
\advance\hoffset by -0.5 truecm
\usepackage{amsmath,amscd}
\usepackage{amssymb}
\usepackage{amsthm}
\newtheorem{Theorem}{Theorem}[section]

\newtheorem{Lemma}[Theorem]{Lemma}

\newtheorem{Proposition}[Theorem]{Proposition}

\newtheorem{Definition}[Theorem]{Definition}

\newtheorem{Remark}[Theorem]{Remark}

\def\V{\mbox{Var}}

\def\Z{{\mathbb Z}}
\def\R\re
\def\V{\bf V}

\def \re{{\mathbb R}}

\def \C{{\mathbb C}}

\def \V{{\bf V}}

\def \X{X_{\perp}}
\def \H{\mathcal H}
\newcommand{\id}{\mathrm{Id}}

\newcommand{\abs}[1]{\lvert #1 \rvert}
\newcommand{\norm}[1]{\lVert #1 \rVert}

\begin{document}
\title[Range of the attenuated ray transform]{On the range of the attenuated ray transform for unitary connections}

\author[G.P. Paternain]{Gabriel P. Paternain}
\address{ Department of Pure Mathematics and Mathematical Statistics,
University of Cambridge,
Cambridge CB3 0WB, UK}
\email {g.p.paternain@dpmms.cam.ac.uk}

\author[M. Salo]{Mikko Salo}
\address{Department of Mathematics and Statistics, University of Jyv\"askyl\"a}
\email{mikko.j.salo@jyu.fi}

\author[G. Uhlmann]{Gunther Uhlmann}
\address{Department of Mathematics, University of Washington and Fondation des
Sciences Math\'ematiques de Paris}

\email{gunther@math.washington.edu}




\begin{abstract} We describe the range of the attenuated ray transform of a unitary connection on a simple surface acting on functions and 1-forms. We use this description to determine the range of the ray transform acting on symmetric tensor fields.

\end{abstract}

\maketitle

\section{Introduction}

In this paper we consider range characterizations for certain ray transforms. The most basic example of the kinds of transforms studied in this paper is the X-ray (or Radon) transform in $\re^2$, which encodes the integrals of a function $f$ in the Schwartz space $\mathcal{S}(\re^2)$ over straight lines:
$$
Rf(s,\omega) = \int_{-\infty}^{\infty} f(s\omega + t\omega^{\perp}) \,dt, \quad s \in \re,\;\;\omega \in S^1.
$$
Here $\omega^{\perp}$ is the rotation of $\omega$ by $90$ degrees counterclockwise. The properties of this transform are classical and well studied \cite{Hel}. One of the basic properties is the characterization of the range of $R$, which describes all possible functions $Rf$ for $f \in \mathcal{S}(\re^2)$ in terms of smoothness, decay and symmetry properties (most importantly, moments with respect to $s$ should be homogeneous polynomials in $\omega$). We refer to \cite{Hel} for the precise statement.

The X-ray transform forms the basis for many imaging methods such as CT and PET in medical imaging. A number of imaging methods involve generalizations of this transform. One such method is SPECT imaging, which is related to the X-ray transform with exponential attenuation factors. In seismic and ultrasound imaging one encounters ray transforms where the measurements are given by integrals over more general families of curves, often modeled as the geodesics of a Riemannian metric. Moreover, integrals of vector fields or other tensor fields instead of just integrals of functions over geodesics may arise, and these transforms are also useful in rigidity questions in differential geometry. See the survey papers \cite{Finch}, \cite{Paternain_survey}, \cite{SU} for more details and background.

The characterization of the range is a basic question for all transforms of this type. In this paper we examine the range of the attenuated ray transform on a simple surface when the attenuation is given by a unitary connection. As an application we describe the range of the ray transform acting on symmetric tensor fields. The notation of Riemannian geometry will be used throughout the paper.

Let $(M,g)$ be a compact oriented Riemannian surface with smooth boundary.  We shall assume that $(M,g)$ is {\it simple}, that is, its boundary is strictly convex and any two points are joined by a unique geodesic depending smoothly on the end points.

We wish to consider unitary connections on the trivial vector bundle $M \times \C^n$. For us a unitary connection means a skew-Hermitian $n\times n$ matrix $A$ whose entries are smooth 1-forms on $M$. Equivalently, we may regard $A$ as a smooth function $A:TM\to \mathfrak{u}(n)$, where $\mathfrak{u}(n)$ is the Lie algebra of the unitary group $U(n)$ and $A$ depends linearly in $v\in T_{x}M$ for each $x\in M$. The connection $A$ induces a covariant derivative $d_A$ acting on sections $s \in C^{\infty}(M, \C^n)$ by $d_A s = ds + As$. The fact that $A$ is unitary means that it respects the sesquilinear inner product of sections in the sense that 
$$
d(s_1, s_2) = (d_A s_1, s_2) + (s_1, d_A s_2).
$$
We refer to \cite{PSU} for more background on unitary connections and the related ray transforms.

Let $SM = \{(x,v) \in TM \,;\, \abs{v} = 1 \}$ be the unit tangent bundle. The geodesics going from $\partial M$ into $M$ can be parametrized by the set $\partial_+ (SM) := \{(x,v) \in SM \,;\, x \in \partial M, \langle v,\nu \rangle \leq 0 \}$ where $\nu$ is the outer unit normal vector to $\partial M$. We also let $\partial_- (SM) := \{(x,v) \in SM \,;\, x \in \partial M, \langle v,\nu \rangle \geq 0 \}$ and denote by $X$ the vector field associated with the geodesic flow $\phi_t$ acting on $SM$.

Given $f\in C^{\infty}(SM,\C^n)$, consider the following transport equation for $u: SM \to \C^n$,
$$Xu + Au = -f \ \ \text{in $SM$}, \quad u|_{\partial_-(SM)} = 0.
$$ Here $A(x,v)$ (the restriction of $A$ to $SM$) acts on functions on $SM$ by multiplication. (Below we will consider $A$ both as a matrix of $1$-forms and as a function on $SM$, depending on the situation.) On a fixed geodesic the transport equation becomes a linear system of ODEs with zero initial condition, and therefore this equation has a unique solution $u=u^f$. 

\begin{Definition}The geodesic ray transform of $f \in C^{\infty}(SM,\C^n)$ with attenuation determined by $A$ is given by$$I_{A} (f) := u^f|_{\partial_+(SM)}.$$ 
\end{Definition}

We note that $I_{A}$ acting on sums of $0$-forms and $1$-forms always has a nontrivial kernel, since 
$$
I_{A}(Xp+Ap) = 0 \text{ for any $p \in C^{\infty}(M,\C^n)$ with $p|_{\partial M} = 0$}.
$$
Thus from the ray transform $I_{A}(f)$ one only expects to recover $f$ up to an element having this form. 
In fact,  in \cite{PSU} we proved the following injectivity result:

\begin{Theorem}\cite{PSU} Let $M$ be a compact simple surface. Assume that $f:SM\to\C^n$ is a smooth function of the form
$F(x)+\alpha_{j}(x)v^j$, where $F:M\to\C^n$ is a smooth function
and $\alpha$ is a $\C^n$-valued 1-form. Let also $A: TM \to \mathfrak{u}(n)$ be a unitary connection. If $I_{A}(f)=0$, then
$F =0$ and $\alpha=dp+Ap$, where $p:M\to\C^n$ is a smooth
function with $p|_{\partial M}=0$.
\label{thm:injective_A}
\end{Theorem}

Based on this result we shall describe the functions in $C^{\infty}(\partial_+ (SM),\C^n)$ which are in the range of $I_{A}^0$ ($I_{A}$ acting on 0-forms) and $I_{A}^{1}$ ($I_{A}$ acting on 1-forms) very much in the spirit of \cite{pestovuhlmann_imrn}, in which the unattenuated case $A=0$ was considered. The description of the range involves the following boundary data:

\begin{enumerate}
\item the scattering relation $\alpha: \partial_+(SM) \to \partial_-(SM)$ which maps a starting point and direction of a geodesic to the end point and direction. If $(M,g)$ is simple, then knowing $\alpha$ is equivalent to knowing the boundary distance function which encodes the distances between any pair of boundary points \cite{Mi}. On two dimensional simple manifolds, the boundary distance function determines the metric up to an isometry which fixes the boundary \cite{PU};
\item the scattering data of the connection $A$ (see Section \ref{sec:prelim});
\item the fibrewise Hilbert transform at the boundary (see Section \ref{sec:prelim}).
\end{enumerate}

Let $\mathcal S^{\infty}(\partial_{+}(SM),\C^n)$ denote the set of those $w\in C^{\infty}(\partial_{+}(SM),\C^n)$ with the property that the unique solution $w^{\sharp}$ to the transport equation 
\[Xw^{\sharp}+Aw^{\sharp}=0,\;\;\;\;\; w^{\sharp}|_{\partial_{+}(SM)}=w\]
is smooth. In Section \ref{sec:range} we define operators
$$P_{\pm}:\mathcal S^{\infty}(\partial_{+}(SM),\C^n)\to C^{\infty}(\partial_{+}(SM),\C^n)$$
exclusively in terms of the boundary data enumerated above and using these operators we can formulate our main result as follows (see Theorem \ref{thm:range_unitary} below):

\begin{Theorem} Let $(M,g)$ be a simple surface and $A$ a unitary connection. Then
\begin{enumerate}
\item A function $u\in C^{\infty}(\partial_{+}(SM),\C^n)$ belongs to the range of $I_{A}^{0}$ if and only if $u=P_{-}w$ for $w\in \mathcal S^{\infty}(\partial_{+}(SM),\C^n)$.
\item A function $u\in C^{\infty}(\partial_{+}(SM),\C^n)$ belongs to the range of $I_{A}^{1}$ if and only if $u=P_{+}w+I_{A}^{1}(\eta)$ for $w\in \mathcal S^{\infty}(\partial_{+}(SM),\C^n)$ and $\eta\in \mathfrak{H}_{A}$.
\end{enumerate}
\end{Theorem}

Here $\mathfrak H_{A}$ denotes the finite dimensional space of 1-forms $\alpha$ such that $d_{A}\alpha=d^*_{A}\alpha=0$ and $j^*\alpha=0$, where $j:\partial M\to M$ is the inclusion map. This space (which is trivial if $A=0$) can also be expressed in terms of a boundary value problem for a suitable Cauchy-Riemann operator with totally real boundary conditions (see Section \ref{sec:proofmain}).

One of the main motivations for considering the range of $I_{A}^{0}$ is to describe in turn the range of the (unattenuated) ray transform acting on symmetric tensor fields. The invertibility of this transform on simple surfaces was recently proved in \cite{PSU_tensor}, which includes more details on the following facts. Given a symmetric (covariant) $m$-tensor field $f = f_{i_1 \cdots i_m} \,dx^{i_1} \otimes \cdots \otimes \,dx^{i_m}$ on $M$, we define the corresponding function on $SM$ by 
$$
f(x,v) = f_{i_1 \cdots i_m} v^{i_1} \cdots v^{i_m}.
$$
The ray transform of $f$ is defined by 
$$
I(f)(x,v) = \int_0^{\tau(x,v)} f(\phi_t(x,v)) \,dt, \quad (x,v) \in \partial_+(SM),
$$
where $\tau(x,v)$ is the exit time of the geodesic determined by $(x,v)$.
If $h$ is a symmetric $(m-1)$-tensor field, its inner derivative $dh$ is a symmetric $m$-tensor field defined by $dh=\sigma\nabla h$, where $\sigma$ denotes
symmetrization and $\nabla$ is the Levi-Civita connection. It is easy to see that
$$
dh(x,v)= Xh(x,v).
$$
If additionally $h|_{\partial M} = 0$, then clearly $I(dh) = 0$. The transform $I$ is said to be \emph{$s$-injective} if these are the only elements in the kernel. The terminology arises from the fact that
any tensor field $f$ may be written uniquely as $f=f^s+dh$, where $f^s$
is a symmetric $m$-tensor with zero divergence and $h$ is an $(m-1)$-tensor
with $h|_{\partial M} = 0$ (cf. \cite{Sh}). The tensor fields $f^s$ and
$dh$ are called respectively the {\it solenoidal} and {\it potential} parts
of $f$. Saying that $I$ is $s$-injective is saying precisely that
$I$ is injective on the set of solenoidal tensors.

In \cite{PSU_tensor} we proved that when $(M,g)$ is a simple surface, then $I$ is $s$-injective for any $m$.  The main idea for one of the two proofs presented in \cite{PSU_tensor} consists in reducing the problem to the case of an attenuated ray transform $I_{A}^{0}$ of a suitable unitary connection $A$ with $n=1$ (scalar case). Using this idea we describe the range of $I$ on symmetric $m$-tensors using the following data.
Since a simple surface is topologically a disk and inherits a complex structure from the metric we can choose a smooth nowhere vanishing section
$\xi$ of the canonical line bundle. As we shall explain in Section \ref{sec:tensor} this section $\xi$ together with the metric $g$ determine naturally a unitary connection $A_{\xi,g}$ with $n=1$. The range of $I$ will be described by the scattering relation $\alpha$, the scattering data of $A_{\xi,g}$ and the fibrewise Hilbert transform at the boundary. For a precise statement see Theorem \ref{thm:rangetensor} below. 

A brief description of the contents of this paper is as follows. In Section \ref{sec:prelim} we establish some preliminaries and background.
Section \ref{sec:range} states carefully the main theorem in this paper which is Theorem \ref{thm:range_unitary}. The application to symmetric tensors is discussed in Section \ref{sec:tensor}. Section \ref{sec:adj} deals with adjoints of the attenuated ray transform and surjectivity properties.
Finally Section \ref{sec:proofmain} contains the proof of Theorem \ref{thm:range_unitary}.

\subsection*{Acknowledgements}
We thank Gareth Ainsworth for discussions related to Theorem \ref{thm:surjectiveon1forms} below.

M.S. was supported in part by the Academy of Finland and an ERC starting grant, and G.U. was partly supported by NSF and a Walker Family Endowed Professorship. G.P.P. thanks the University of Washington and the University of Jyv\"askyl\"a for hospitality while this work was being carried out.

\section{Preliminaries}\label{sec:prelim}

In this section we collect various facts which will be used later on.  We also introduce the scattering data of a unitary connection and we review a regularity result for the transport equation obtained in \cite{PSU}. See \cite{PSU}, \cite{PSU_tensor} for further details. In what follows $(M,g)$ is a compact orientable simple surface and $A$ is a unitary connection as defined in the Introduction.

\subsection{Scattering data of a unitary connection} As before, $X$ denotes the vector field associated with the geodesic flow
$\phi_t$ acting on $SM$ and we look at the unique solution $U_{A}:SM\to U(n)$ of

\begin{equation}
\left\{\begin{array}{ll}
XU_{A}+A(x,v)U_{A}=0,\;(x,v)\in SM\\
U_{A}|_{\partial_{+}(SM)}=\id.\\
\end{array}\right.
\label{eq:4}
\end{equation}
The scattering data of the connection $A$ is the map $C_{A}:\partial_{-}(SM)\to U(n)$ defined as $C_{A}:=U_{A}|_{\partial_{-}(SM)}$. It is easy to check that the scattering data has the following gauge invariance: if $G:M\to U(n)$ is a smooth map such that $G|_{\partial M}=\mbox{\rm Id}$, then the unitary connection $G^{-1}dG+G^{-1}AG$ has the same scattering data as $A$.
The scattering data $C_{A}$ determines $A$ up to this gauge equivalence \cite{PSU}.

Similarly, we can consider the unique matrix solution of 
\[\left\{\begin{array}{ll}
XW_{A}+A(x,v)W_{A}=0,\;(x,v)\in SM\\
W_{A}|_{\partial_{-}(SM)}=\id\\
\end{array}\right.\]
and this gives a corresponding scattering map $D_{A}:\partial_{+}(SM)\to U(n)$ defined as $D_{A}:=W_{A}|_{\partial_{+}(SM)}$. However, it is straightforward to check that $C_{A}$ and $D_{A}$ are related by the scattering relation $\alpha$ of the metric $g$: 
$$
C_{A}^{-1}\circ\alpha=D_{A}.
$$
Recall that $\alpha$ may be extended as a $C^{\infty}$ diffeomorphism of $\partial(SM)$ \cite{PU}.

Using the fundamental solution $U_{A}$ it is possible to give an integral representation for
$I_{A}$ acting on functions $f\in C^{\infty}(SM,\C^n)$.  To make this more transparent consider first the case $n=1$.  By explicitly solving the transport equation along the geodesic determined by $(x,v)\in\partial_{+}(SM)$,
\[\dot{u}+A(\phi_{t}(x,v))u=-f(\phi_{t}(x,v))\]
with $u(\tau(x,v))=0$ we find that
\[I_{A}(f)(x,v)=u(0)= \int_0^{\tau(x,v)} f(\phi_t(x,v)) \text{exp}\left[ \int_0^t A(\phi_s(x,v)) \,ds \right] dt.\]
In the scalar case $n=1$ we also see that if $(x,v)\in\partial_{+}(SM)$, then 
$$U_{A}(\phi_{t}(x,v))=\text{exp}\left[ -\int_0^t A(\phi_s(x,v)) \,ds \right] \in S^{1}=U(1).$$
In general ($n\geq 2$) a simple calculation shows that the following integral formula holds:
\begin{equation}\label{eq:xray}
I_{A} (f)(x,v) = \int_{0}^{\tau(x,v)} U_{A}^{-1}(\phi_{t}(x,v))f(\phi_{t}(x,v))\,dt.
\end{equation}

\subsection{Geometry of $SM$ and the fibrewise Hilbert transform}
Since $M$ is assumed oriented there is a circle action on the fibres of $SM$ with infinitesimal generator $V$ called the {\it vertical vector field}. It is possible to complete the pair $X,V$ to a global frame
of $T(SM)$ by considering the vector field $X_{\perp}:=[X,V]$, where $[\cdot,\cdot]$ stands for the Lie bracket or commutator of two vector fields. There are two additional structure equations given by $X=[V,X_{\perp}]$ and $[X,X_{\perp}]=-KV$
where $K$ is the Gaussian curvature of the surface. Using this frame we can define a Riemannian metric on $SM$ by declaring $\{X,X_{\perp},V\}$ to be an orthonormal basis and the volume form of this metric will be denoted by $d\Sigma^3$ and referred to as the {\it usual} volume form on $SM$. The fact that $\{ X, X_{\perp}, V \}$ are orthonormal together with the commutator formulas implies that the Lie derivative of $d\Sigma^3$ along the three vector fields vanishes, and consequently these vector fields are volume preserving.

Given functions $u,v:SM\to \C^n$ we consider the
inner product
\[(u,v) =\int_{SM} (u,v)_{\C^n}\,d\Sigma^3.\]
The space $L^{2}(SM,\C^n)$ decomposes orthogonally
as a direct sum
\[L^{2}(SM,\C^n)=\bigoplus_{k\in\mathbb Z}H_{k}\]
where $-iV$ acts as $k\,\mbox{\rm Id}$ on $H_k$.
 Let $\Omega_{k}:=C^{\infty}(SM,\C^n)\cap H_{k}$. A smooth function $u:SM\to\C^n$ has a Fourier series expansion
\[u=\sum_{k=-\infty}^{\infty}u_{k}.\]

Following Guillemin and Kazhdan in \cite{GK} we introduce the following
first order elliptic operators 
$$\eta_{+},\eta_{-}:C^{\infty}(SM,\C^n)\to
C^{\infty}(SM,\C^n)$$ given by
\[\eta_{+}:=(X+iX_{\perp})/2,\;\;\;\;\;\;\eta_{-}:=(X-iX_{\perp})/2.\]
Clearly $X=\eta_{+}+\eta_{-}$. The commutation relations $[-iV,\eta_{+}]=\eta_{+}$ and
$[-iV,\eta_{-}]=-\eta_{-}$ imply that
\[\eta_{+}:\Omega_{k}\to \Omega_{k+1},\;\;\;\;\eta_{-}:\Omega_{k}\to \Omega_{k-1}.\]
These operators will be useful in Section \ref{sec:proofmain}.

 The description of the range of $I_{A}$ makes use of the fibrewise {\it Hilbert transform} $\H$.
This can be introduced in various ways (cf.~\cite{PU,SaU}), but here we simply indicate that it acts fibrewise and
 for $u_{k}\in \Omega_k$,
\[\H(u_{k})=-\mbox{\rm sgn}(k)\,iu_{k}\]
where we use the convention $\mbox{\rm sgn}(0)=0$.
Moreover, $\H(u)=\sum_{k}\H(u_{k})$.
Observe that
\[(\id+i\H)u=u_0+2\sum_{k=1}^{\infty}u_{k},\]
\[(\id-i\H)u=u_0+2\sum_{k=-\infty}^{-1}u_{k}.\]
The following bracket relation (cf. \cite{PSU}) extends the bracket in \cite{PU}:

\begin{equation}\label{eq:bracket}
[\H,X+A]u=(\X+\star A)(u_{0})+\{(\X+\star A)(u)\}_{0},
\end{equation}
where $\star$ denotes the Hodge star operator of the metric $g$ and $\{(\X+\star A)(u)\}_{0}$ indicates the zero Fourier component
of $(X_{\perp}+\star A)(u)$.

\subsection{Regularity results for the transport equation}

Given a smooth $w\in C^{\infty}(\partial_{+}(SM),\C^n)$ consider
the unique solution $w^{\sharp}:SM\to \C^n$
to the transport equation:
\[\left\{\begin{array}{ll}
X(w^{\sharp})+Aw^{\sharp}=0,\\
w^{\sharp}|_{\partial_{+}(SM)}=w.\\
\end{array}\right.\] 
Observe that
\[w^{\sharp}(x,v)=U_{A}(x,v)w(\alpha\circ\psi(x,v))\,\]
where $\psi(x,v):=\phi_{\tau(x,v)}(x,v)$ (recall that $\phi_{t}$ is the geodesic flow and $\tau(x,v)$ is the time it takes the geodesic determined by $(x,v)$ to exit $M$).
If we introduce the operator 
\[Q:C(\partial_{+}(SM),\C^n)\to C(\partial(SM),\C^n)\]
by setting
\[Qw(x,v)=\left\{\begin{array}{ll}
w(x,v)&\mbox{\rm if}\;(x,v)\in\partial_{+}(SM)\\
C_{A}(x,v)(w\circ\alpha)(x,v)&\mbox{\rm if}\;(x,v)\in\partial_{-}(SM),\\
\end{array}\right.\] 
then
\[w^{\sharp}|_{\partial(SM)}=Qw.\]
Define
\[\mathcal S^{\infty}(\partial_{+}(SM),\C^n):=\{w\in C^{\infty}(\partial_{+}(SM),\C^n):\;w^{\sharp}\in C^{\infty}(SM,\C^n)\}.\]

We have

\begin{Lemma}[\cite{PSU}] The set of those smooth $w$ such that $w^{\sharp}$ is smooth
is given by 
\[\mathcal S^{\infty}(\partial_{+}(SM),\C^n)=\{w\in C^{\infty}(\partial_{+}(SM),\C^n):\;Qw\in C^{\infty}(\partial(SM),\C^n)\}.\]
\label{lemma:ss}
\end{Lemma}

\section{Description of the range}\label{sec:range}

In this section we describe the main result of this paper (Theorem \ref{thm:range_unitary} below). Let $(M,g)$ be a simple surface and $A$ a unitary connection. We shall consider $I_{A}$ acting on $0$-forms and 1-forms separately and for this we shall use the notation $I_{A}^{0}$ and $I_{A}^{1}$ respectively (as in the Introduction).

Suppose $f=(X+A)g$ for some smooth function $g$. Then clearly
$X(u^f+g)+A(u^f+g)=0$, where $u^f$ solves $Xu+Au=-f$ with $u|_{\partial_{-}(SM)}=0$.  Since $(u^f+g)|_{\partial_{-}(SM)}=g|_{\partial_{-}(SM)}$ we deduce that $(u^f+g)|_{\partial_{+}(SM)}=(C_{A}^{-1}g)\circ\alpha$. In other words
\begin{equation}
I_{A}((X+A)g)=[(C_{A}^{-1}g)\circ\alpha-g]|_{\partial_{+}(SM)}.
\label{eq:ftc}
\end{equation}
Motivated by this we introduce the operator
\[B:C(\partial(SM),\C^n)\to C(\partial_{+}(SM),\C^n)\]
defined by
\[Bg:=[(C_{A}^{-1}g)\circ\alpha-g]|_{\partial_{+}(SM)}\]
so we can write (\ref{eq:ftc}) as
\begin{equation}
I_{A}((X+A)g)=Bg^0,\;\;g^0=g|_{\partial(SM)}.
\label{eq:ftc2}
\end{equation}

Let $\H_{\pm}u:=\H u_{\pm}$, where $u_{+}$ (resp. $u_{-}$) denotes the even (resp. odd) part of $u$ with respect to the $v$ variable. In terms of Fourier coefficients, $u_+$ consists of the Fourier components of $u$ having even degree and $u_-$ of the components having odd degree.

\begin{Definition} Let $P_{\pm}:\mathcal S^{\infty}(\partial_{+}(SM),\C^n)\to C^{\infty}(\partial_{+}(SM),\C^n)$ be the operators defined by $P_{\pm}=B\H_{\pm}Q$.
\end{Definition}
Note that the operators $P_{\pm}$ are completely determined by the scattering data $(\alpha,C_{A})$
and the metric at the boundary.

The connection $A$ induces an operator $d_{A}$ acting on $\C^n$-valued smooth differential forms on $M$ by the formula $d_{A}\alpha=d\alpha+A\wedge \alpha$.  We let
$d_{A}^*$ be its natural adjoint with respect to the usual $L^2$ inner product of forms. We have $d^*_{A}=-\star d_{A}\star$.
We let $\mathfrak H_{A}$ denote the space of all 1-forms $\alpha$ with $d_{A}\alpha=d_{A}^*\alpha=0$ and $j^*\alpha=0$ where $j:\partial M\to M$ is the inclusion map. We call these forms {\it $A$-harmonic}. 
This is a finite dimensional space since we are dealing with an elliptic system
with a regular boundary condition (cf. \cite[Section 5.11]{Ta}).
Here is our main result:

\begin{Theorem} \label{thm:range_unitary} Let $(M,g)$ be a simple surface and $A$ a unitary connection. Then
\begin{enumerate}
\item A function $u\in C^{\infty}(\partial_{+}(SM),\C^n)$ belongs to the range of $I_{A}^{0}$ if and only if $u=P_{-}w$ for $w\in \mathcal S^{\infty}(\partial_{+}(SM),\C^n)$.
\item A function $u\in C^{\infty}(\partial_{+}(SM),\C^n)$ belongs to the range of $I_{A}^{1}$ if and only if $u=P_{+}w+I_{A}^{1}(\eta)$ for $w\in \mathcal S^{\infty}(\partial_{+}(SM),\C^n)$ and $\eta\in \mathfrak{H}_{A}$.
\end{enumerate}
\end{Theorem}

When $A=0$ we have $\mathfrak{H}_{A}=0$ since $M$ is a disk. We will show below
that $\mathfrak{H}_{A}$ may also be expressed in terms of a suitable Cauchy-Riemann operator
with totally real boundary conditions.

\section{Range of the ray transform acting on symmetric tensors}\label{sec:tensor}

In this section we apply  Theorem \ref{thm:range_unitary} to obtain a description of the range of the ray transform acting on symmetric tensors.

Let $(M,g)$ be a simple surface. The metric $g$ induces a complex structure on $M$ and we let $\kappa$ be the canonical line bundle (which we may identify with $T^*M$).
Recall that $H_m$ ($m\in\Z$) is the set of functions in $f\in L^{2}(SM,\C)$ such that
$Vf=imf$. 
The set $\Omega_m=H_{m}\cap C^{\infty}(SM,\C)$ can be identified with the set $\Gamma(M,\kappa^{\otimes m})$ of smooth sections of $m$-th tensor power of the canonical line bundle $\kappa$. This identification depends on the metric and is explained in detail in \cite[Section 2]{PSUa}, but let us give a brief description of it.  Given a section $\xi\in \Gamma(M,\kappa^{\otimes m})$ we can obtain a function on $\Omega_m$ simply by restriction to $SM$: $\xi$ determines the function $SM\ni (x,v)\mapsto \xi_{x}(v^{\otimes m})$ and this gives a 1-1 correspondence. 

Since $M$ is a disk, there is $\xi\in  \Gamma(M,\kappa)$ which is nowhere vanishing. Having picked this section we may define a function 
$h:SM\to S^{1}$ by setting $h(x,v)=\xi_{x}(v)/|\xi_{x}(v)|$.  By construction $h\in\Omega_{1}$.  Our description of the range will be based on this choice of $h$.
Define
$$
A_{\xi,g}=A := -h^{-1} Xh.
$$
Observe that since $h\in\Omega_1$, then $h^{-1}=\bar{h}\in\Omega_{-1}$. Also $Xh=\eta_{+}h+\eta_{-}h\in\Omega_{2}\oplus\Omega_{0}$ which implies that $A\in\Omega_{1}\oplus\Omega_{-1}$. It follows that $A$ is the restriction to $SM$ of a purely imaginary 1-form on $M$. Hence we have a unitary connection for the scalar case $n=1$.

First we will describe the range of the geodesic ray transform $I$ restricted to $\Omega_{m}$:
$$I_{m}:=I|_{\Omega_{m}}:\Omega_{m}\to C^{\infty}(\partial_{+}(SM),\C).$$
Observe that if $u$ solves the transport equation $Xu=-f$ with $u|_{\partial_{-}(SM)}=0$, then
$h^{-m}u$ solves $(X-mA)(h^{-m}u)=-h^{-m}f$ and $h^{-m}u|_{\partial_{-}(SM)}=0$.
Also note that $h^{-m}f\in\Omega_{0}$.
Thus
\begin{equation}\label{eq:relation}
I_{-mA}^{0}(h^{-m}f)=\left(h^{-m}|_{\partial_{+}(SM)}\right)I_{m}(f).
\end{equation}
But the range of the left-hand side of (\ref{eq:relation}) was described in Theorem \ref{thm:range_unitary}. We can be a bit more explicit in the scalar case. We have
\[Q_{m}w(x,v)=\left\{\begin{array}{ll}
w(x,v)&\mbox{\rm if}\;(x,v)\in\partial_{+}(SM)\\
(e^{-m\int_{0}^{\tau(x,v)}A(\phi_{t}(x,v))\,dt}w)\circ\alpha(x,v)&\mbox{\rm if}\;(x,v)\in\partial_{-}(SM)\\
\end{array}\right.\] 
and
\[B_{m}g=[e^{m\int_{0}^{\tau(x,v)}A(\phi_{t}(x,v))\,dt}(g\circ\alpha)-g]|_{\partial_{+}(SM)}.\]
In other words:
\[Q_{m}w(x,v)=\left\{\begin{array}{ll}
w(x,v)&\mbox{\rm if}\;(x,v)\in\partial_{+}(SM)\\
(e^{-mI_{1}(A)}w)\circ\alpha(x,v)&\mbox{\rm if}\;(x,v)\in\partial_{-}(SM)\\
\end{array}\right.\] 
and
\[B_{m}g=[e^{mI_{1}(A)}(g\circ\alpha)-g]|_{\partial_{+}(SM)}.\]
As before we set
\[P_{m,-}=B_{m}\H_{-}Q_{m}.\]
Directly from (\ref{eq:relation}) and Theorem \ref{thm:range_unitary} we derive:

\begin{Theorem} \label{thm:im} Let $(M,g)$ be a simple surface. Then a function $u\in C^{\infty}(\partial_{+}(SM),\C)$ belongs to the range of $I_{m}$ if and only if $u=\left(h^{m}|_{\partial_{+}(SM)}\right)P_{m,-}w$ for $w\in \mathcal S_{m}^{\infty}(\partial_{+}(SM),\C)$, where this last space denotes the set of all smooth $w$ such that $Q_{m}w$ is smooth.

\end{Theorem}

Suppose now $F$ is a complex-valued symmetric tensor of order $m$, and denote its restriction to $SM$ by $f$. Recall from \cite[Section 2]{PSU_tensor} that there is a 1-1 correspondence between 
complex-valued symmetric tensors of order $m$ and functions in $SM$ of the form
$f=\sum_{k=-m}^{m}f_k$ where $f_k\in\Omega_k$ and $f_k=0$ for all $k$ odd (resp. even) if
$m$ is even (resp. odd).

Since
\[I(f)=\sum_{k=-m}^{m}I_{k}(f_{k})\]
we deduce  directly from Theorem \ref{thm:im} the following.

\begin{Theorem} Let $(M,g)$ be a simple surface. If $m=2l$ is even, a function 
$u\in C^{\infty}(\partial_{+}(SM),\C)$ belongs to the range of the ray transform acting on complex-valued  symmetric $m$-tensors if and only if there are $w_{2k}\in \mathcal S_{2k}^{\infty}(\partial_{+}(SM),\C)$
such that
\[u=\sum_{k=-l}^{l}\left(h^{2k}|_{\partial_{+}(SM)}\right)P_{2k,-}w_{2k}.\]
Similarly, if $m=2l+1$ is odd, a function 
$u\in C^{\infty}(\partial_{+}(SM),\C)$ belongs to the range of the ray transform acting on complex-valued  symmetric $m$-tensors if and only if there are $w_{2k+1}\in \mathcal S_{2k+1}^{\infty}(\partial_{+}(SM),\C)$
such that
\[u=\sum_{k=-l-1}^{l}\left(h^{2k+1}|_{\partial_{+}(SM)}\right)P_{2k+1,-}w_{2k+1}.\]
\label{thm:rangetensor}

\end{Theorem}

\begin{Remark}{\rm The idea of reducing the study of the (unattenuated) ray transform on $m$-tensors to the study of a suitable attenuated transform on functions was first employed in \cite{PSU_tensor} to solve the tensor tomography problem on surfaces. The same idea can be used to study 
the attenuated ray transform on higher order tensors. The order of the tensor can be lowered by adding a suitable attenuation.
With this idea one could in particular give a description of the range of $I_{A}^{1}$ using just the range of $I_{A}^{0}$ and a section
$\xi$ as above. This allows a new formulation of the second item in Theorem \ref{thm:range_unitary} which does not involve $A$-harmonic forms.
We leave the details of this to the interested reader.

}
\end{Remark}

\section{Adjoints and surjectivity properties}\label{sec:adj}

Let $d\Sigma^3$ be the usual volume form in $SM$ and let $d\Sigma^2$ be the volume form on $\partial(SM)$. When we write
$L^2_{\mu}(\partial_{+}(SM),\C^n)$ we mean that we are considering
the Hermitian inner product of functions (sections) with respect to the measure
$d\mu(x,v)=\langle v,\nu(x)\rangle d\Sigma^2(x,v)$.

\begin{Lemma} The operator $I_{A}$ extends to a bounded operator
\[I_{A}:L^2(SM,\C^n)\to L^2_{\mu}(\partial_{+}(SM),\C^n).\]
\end{Lemma}

\begin{proof} It is easy to see
using compactness of $SM$ and (\ref{eq:xray}) that there exists a constant $c$ independent
of $f$ such that
\[\|I_{A}(f)(x,v)\|^2_{\C^n}\leq c\int_{0}^{\tau(x,v)}\|f(\phi_{t}(x,v))\|^2_{\C^n}\,dt\]
for all $(x,v)\in \partial_{+}(SM)$.
Then
\begin{align*}
\|I_{A}(f)\|_{\mu}^2&=\int_{\partial_{+}(SM)}\|I_{A}(f)(x,v)\|^2_{\C^n}\,d\mu(x,v)\\
&\leq c\int_{\partial_{+}(SM)}\left(\int_{0}^{\tau(x,v)}\|f(\phi_{t}(x,v))\|^2_{\C^n}\,dt\right)\,d\mu(x,v)\\
&=c\int_{SM}\|f\|^2\,d\Sigma^3
\end{align*}
by Santalo's formula (cf. \cite[Appendix A.4]{DPSU}).

\end{proof}

\subsection{Adjoint of $I_{A}$.} Consider the dual
$I_{A}^*:L^2_{\mu}(\partial_{+}(SM),\C^n)\to L^2(SM,\C^n)$ of $I_{A}$. We wish
to find an expression for it. For this consider $h\in L^{2}_{\mu}(\partial_{+}(SM),\C^n)$
and let $h_{\psi}$ be the function defined on $SM$ as follows:
\[h_{\psi}(x,v)=h\circ\alpha\circ\psi,\]
where $\psi(x,v)=\phi_{\tau(x,v)}(x,v)$. In other words $h_{\psi}$ is defined
as being constant on the orbits of the geodesic flow and equal to $h$
on $\partial_{+}(SM)$.
Now we compute using Santalo's formula:
\begin{align*}
(I_{A}(f),h)_{\mu}&=\int_{\partial_{+}(SM)}\left(\int_{0}^{\tau(x,v)}U_{A}^{-1}(\phi_{t}(x,v))f(\phi_{t}(x,v))\,dt,h(x,v)\right)_{\C^n}\,d\mu(x,v)\\
&=\int_{\partial_{+}(SM)}\int_{0}^{\tau(x,v)}( U_{A}^{-1}(\phi_{t}(x,v))f(\phi_{t}(x,v))\,dt,h_{\psi}(\phi_{t}(x,v)))_{\C^n}\,dt\,d\mu(x,v)\\
&=\int_{SM}(U_{A}^{-1}f,h_{\psi})_{\C^n}\,d\Sigma^3\\
&=\int_{SM}( f,(U_{A}^{-1})^{*}h_{\psi})_{\C^n}\,d\Sigma^{3}\\
&=(f,(U_{A}^{-1})^{*}h_{\psi})_{L^{2}}.
\end{align*}
It follows that
\[I_{A}^*(h)=(U_{A}^{-1})^{*}h_{\psi}.\]
Note that if $A$ is unitary this simplifies to
\[I_{A}^*(h)=U_{A}h_{\psi}= h^{\sharp}.\]

\subsection{Adjoints on $H_{m}$} Suppose we now want to consider the attenuated ray transform
acting on special subspaces of $L^{2}(SM,\C^n)$ like $H_{m}$. Let $i_{m}:H_{m}\to L^{2}(SM,\C^n)$ be the
inclusion and 
\[I_{m,A}:=I_{A}\circ i_{m}.\]
Since the adjoint of $i_m$ is the orthogonal projection over $H_{m}$ we deduce that
\[I_{m,A}^*(h)=((U_{A}^{-1})^{*}h_{\psi})_{m}\]
and in the unitary case we have
\begin{equation}\label{eq:adjoint_unitary}
I_{m,A}^*(h)=(h^{\sharp})_{m}
\end{equation}

\begin{Remark} \label{Remark:adjoints}{\rm The space $H_{0}$ may be identified with $L^{2}(M,\C^n)$. The identification
produces a factor of $2\pi$ which accounts for the additional (innocuous) integration
in the fibre direction. Thus strictly speaking $(I_{A}^{0})^*=2\pi I_{0,A}^{*}$. A similar remark applies
to $I_{A}^{1}=I_{1,A}+I_{-1,A}$ and the identification between 1-forms and $\Omega_{-1}\oplus\Omega_{1}$ produces a factor of $\pi$. Hence in the unitary case $(I_{A}^{1})^*|_{SM}=\pi(h^{\sharp}_{-1}+h^{\sharp}_{1})$.}
\end{Remark}

\subsection{The operator $N_{A}=I_{A}^{*}I_{A}$.}
Consider $N_{A}:=I_{A}^{*} I_{A}:L^2(SM,\C^n)\to L^{2}(SM,\C^n)$. Using the expressions derived above for $I_{A}$ and $I_{A}^*$ one obtains:
\begin{align*}
N_{A}(f)(x,v)&=(U_{A}^{-1})^*(x,v)\int_{-\tau(x,-v)}^{\tau(x,v)}
U_{A}^{-1}(\phi_{t}(x,v))f(\phi_{t}(x,v))\,dt.
\end{align*}
Note that if $A$ is unitary this expression simplifies to
\begin{align*}
N_{A}(f)(x,v)&=U_{A}(x,v)\int_{-\tau(x,-v)}^{\tau(x,v)}
U_{A}^{-1}(\phi_{t}(x,v))f(\phi_{t}(x,v))\,dt.
\end{align*}
If we now let $N_{m,A}:=I_{m,A}^{*}I_{m,A}$, then $N_{m,A}=i^{*}_{m}N_{A}i_{m}$ and therefore
\begin{align*}
N_{m,A}(f)(x,v)&=\left((U_{A}^{-1})^*(x,v)\int_{-\tau(x,-v)}^{\tau(x,v)}
U_{A}^{-1}(\phi_{t}(x,v))f(\phi_{t}(x,v))\,dt\right)_{m}.
\end{align*}
In particular
\begin{align*}
N_{0,A}(f)(x)&=\frac{1}{2\pi}\int_{S_{x}}(U_{A}^{-1})^*(x,v)\left(\int_{-\tau(x,-v)}^{\tau(x,v)}
U_{A}^{-1}(\phi_{t}(x,v))f(\pi\circ\phi_{t}(x,v))\,dt\right)dS_{x}.
\end{align*}

The key property for $N_{0,A}$ is:

\begin{Lemma} $N_{0,A}$ is an elliptic classical $\Psi DO$ of order $-1$ in the interior of $M$.
\label{lemma:elliptic}
\end{Lemma}

\begin{proof} The proof of this claim for $n=1$ is provided in \cite[Proposition 2]{FSU}
for a more general type of ray transforms and the arguments used there can be extended to
this case. Another way to prove the proposition is to use
that $I_{0, A}$ is a Fourier integral operator \cite{G} acting on vector bundles. Since the manifold is simple then the Bolker condition is satisfied
\cite{GrU} and $N_{0,A}$ is an elliptic pseudodifferential system of order $-1$.
A similar argument applies to $N_{m,A}$. We remark that this argument is valid on any open simple manifold containing $M.$ 
\end{proof}

\subsection{Surjectivity results for the adjoint}

We will need the following analogue of \cite[Theorem 1.4]{PU} concerning the surjectivity of $I_{0,A}^*$:

\begin{Theorem} Given $b\in C^{\infty}(M,\C^n)$ there exists $w\in {\mathcal S}^{\infty}(\partial_{+}(SM),\C^n)$ such that
$I_{0,A}^*(w)=b$.\label{thm:surjective}
\end{Theorem}

\begin{proof} Given Lemma \ref{lemma:elliptic}, the proof of this theorem
is identical to the proof of Theorem 1.4 in \cite{PU} and is therefore omitted.
The key result needed is injectivity of $I_{0,A}$ on functions which follows from
Theorem \ref{thm:injective_A} in the Introduction: let $f:M\to\C^n$ be a smooth function with $I_{0,A}(f)=0$. 
Then $f=0$.

\end{proof}

We will also need the analogue of \cite[Theorem 4.2]{pestovuhlmann_imrn} and its improvement in \cite[Theorem 2.1]{DU}:

\begin{Theorem} Given a smooth $\C^n$-valued 1-form $\beta$ with $d_{A}\star\beta=0$, there 
 exists $w\in {\mathcal S}^{\infty}(\partial_{+}(SM),\C^n)$ such that
$(I_{A}^{1})^*(w)=\beta$.\label{thm:surjectiveon1forms}
\end{Theorem}

\begin{proof} The proof presented here is different from those in \cite[Theorem 4.2]{pestovuhlmann_imrn} and \cite[Theorem 2.1]{DU} and is exclusively based on Theorem \ref{thm:surjective} and the same ideas used in Section \ref{sec:tensor} to describe the range of the ray transform on tensors. As before consider the purely imaginary 1-form
$$
a:=A_{\xi,g}= -h^{-1} Xh.
$$
where $h\in\Omega_1$ is nowhere vanishing. Next observe that if $u:SM\to\C^n$ is any smooth function then
\begin{equation}
(X+A-ma\mbox{\rm Id})u=h^{-m}((X+A)(h^m u))
\label{eq:att-t}
\end{equation}
where $m\in \Z$. First we show the following result which is interesting in its own right:

\begin{Lemma} Given any $f\in\Omega_m$, there exists $w\in C^{\infty}(SM,\C^n)$ such that
\begin{enumerate}
\item $(X+A)w=0$,
\item $w_m=f$.
\end{enumerate}
\label{lemma:sobre}
\end{Lemma}

\begin{proof} Since $A-ma\mbox{\rm Id}$ is a unitary connection we may apply Theorem \ref{thm:surjective} to deduce that there is $u\in C^{\infty}(SM,\C^n)$ such that
$(X+A-ma\mbox{\rm Id})u=0$ and $u_0=h^{-m}f$. If we let $w:=h^{m}u$, then clearly
$w_m=f$ and by (\ref{eq:att-t}) we also have $(X+A)w=0$.
\end{proof}

Introduce the operators $\mu_{\pm}=\eta_{\pm}+A_{\pm 1}$. Clearly $X+A=\mu_{+}+\mu_{-}$
and it is easy to check that $d_{A}\star \beta=0$ if and only if 
\begin{equation}
\mu_{+}(\beta_{-1})+\mu_{-}(\beta_{1})=0
\label{eq:div}
\end{equation}
where $\beta=\beta_{-1}+\beta_{1}$ (this follows right away from Lemma \ref{lemma:dAf} below).

By Lemma \ref{lemma:sobre} we may find odd functions $p,q\in C^{\infty}(SM,\C^n)$ solving the transport equation $(X+A)p=(X+A)q=0$ and with $p_{-1}=\beta_{-1}$ and $q_{1}=\beta_{1}$.
Consider now the smooth function
\[w:=\sum_{-\infty}^{-1}p_k+\sum_{1}^{\infty}q_k.\]
Clearly $w_{-1}+w_{1}=\beta_{-1}+\beta_{1}=\beta$ and thanks to (\ref{eq:div}) we also
have $(X+A)w=0$. The existence of such a solution of the transport equation is certainly equivalent to the statement of Theorem \ref{thm:surjectiveon1forms} and hence its proof is completed.

\end{proof}

\section{Proof of Theorem \ref{thm:range_unitary}}\label{sec:proofmain}

The first step is to show the following lemma:

\begin{Lemma} \label{lemma:CR}  Let $(M,g)$ be a Riemannian disk and let $A$ be a unitary connection.
\begin {enumerate}
\item Let $\alpha$ be a smooth $\C^n$-valued 1-form. Then there are functions $a,p\in C^{\infty}(M,\C^n)$ and $\eta\in\mathfrak{H}_{A}$
such that $p|_{\partial M}=0$ and $d_{A}p+\star d_{A}a+\eta=\alpha$.

\item Given $b\in C^{\infty}(M,\C^n)$ there is a smooth $\C^n$-valued 1-form $\beta$ with $\star d_{A}\beta=b$
and $d_{A}\star\beta=0$.
\end{enumerate}

\end{Lemma}

We defer the proof of this lemma to the end of the section. Assuming Lemma \ref{lemma:CR}, the next step in the proof of Theorem  \ref{thm:range_unitary} is to apply the bracket formula (\ref{eq:bracket}) to any smooth solution
$w^{\sharp}$ of the transport equation $(X+A)w^{\sharp}=0$ to obtain
\[-(X+A)\H w^{\sharp}=(X_{\perp}+\star A)(w^{\sharp}_{0})+\{(X_{\perp}+\star A)(w^{\sharp})\}_{0}.\]
If $f\in\Omega_{0}$, then $X_{\perp}f=\star df$ and therefore 
$$
(X_{\perp}+\star A)(w^{\sharp}_{0}) = \star d_A(w^{\sharp}_{0}).
$$
For the other term, since $X_{\perp}=i(\eta_{-}-\eta_{+})$ and $\star(A_1+A_{-1}) = i(A_{-1}-A_1)$ we have
\begin{align*}\{(X_{\perp}+\star A)(w^{\sharp})\}_{0}&=i(\eta_{-}(w^{\sharp}_{1})-\eta_{+}(w^{\sharp}_{-1}))+i(A_{-1}w^{\sharp}_{1}-A_{1}w^{\sharp}_{-1})\\
&=i(\mu_{-}(w^{\sharp}_{1})-\mu_{+}(w^{\sharp}_{-1})),
\end{align*}
where $\mu_{\pm}:=\eta_{\pm}+A_{\pm 1}$. The last expression may be simplified as follows:

\begin{Lemma} For any $\C^n$-valued 1-form $\alpha = \alpha_1 + \alpha_{-1}$ we have
\[\star d_{A}\alpha=2i(\mu_{-}(\alpha_{1})-\mu_{+}(\alpha_{-1})).\]
\label{lemma:dAf}
\end{Lemma}

\begin{proof} This is just a calculation. Let $v \in T_x M$ be a unit vector and let $iv \in T_x M$ be the unique unit vector such that $\{v, iv\}$ is an oriented orthonormal basis of $T_x M$. First we see that
\[d\alpha_{x}(v,iv)=X_{\perp}(\alpha)-X(\star\alpha)\]
and
\[\star (A\wedge\alpha)_{x}=(A\wedge\alpha)_{x}(v,iv)=-A_{x}(v)\star\alpha_{x}(v)+\star A_{x}(v)\alpha_{x}(v).\]
Hence
\[\star d_{A}\alpha=(X_{\perp}+\star A)(\alpha)-(X+A)(\star\alpha).\]
Now observe that $\alpha_{\pm 1}=(\alpha\pm i\star \alpha)/2$ and similarly for $A$.
Finally, using the definition of $\mu_{\pm}$ one checks the identity of the lemma.
\end{proof}

Putting everything together and using the formulas for the adjoint of the ray transform in Remark \ref{Remark:adjoints}, we have
\begin{align*}
-2\pi(X+A)\H w^{\sharp} &= 2\pi\star d_A(w^{\sharp}_{0}) + \pi \star d_A(w^{\sharp}_1+w^{\sharp}_{-1}) \\
 &= \star d_{A}(I^{0}_{A})^*(w)+\star d_{A}(I^{1}_{A})^*(w).
\end{align*}
Splitting into even and odd parts, we obtain 
\begin{align*}
-2\pi(X+A)\H_- w^{\sharp} &= \star d_{A}(I^{1}_{A})^*(w), \\
-2\pi(X+A)\H_+ w^{\sharp} &= \star d_{A}(I^{0}_{A})^*(w).
\end{align*}
If we now apply $I_{A}$ to both sides of these identities and recall the definitions of $P_{\pm}$ and $Q$, together with the fact that $I_A((X+A)g) = B (g|_{\partial(SM)})$, we have proved the following result.

\begin{Proposition} \label{prop:factorization} The following equalities hold:
\[-2\pi P_{-}=I_{A}^{0}\star d_{A}(I^{1}_{A})^*,\]
\[-2\pi P_{+}=I_{A}^{1}\star d_{A}(I^{0}_{A})^*.\]
\end{Proposition}
From these factorization identities the proof of Theorem \ref{thm:range_unitary} easily follows
when combined with Lemma \ref{lemma:CR} and the surjectivity results in Section \ref{sec:adj}.

\begin{proof}[Proof of Theorem \ref{thm:range_unitary}]
Let us prove the first claim. Suppose that $u=P_{-}w$, then the first identity in Proposition \ref{prop:factorization} shows that $u$ belongs to the range of $I^{0}_{A}$. Conversely, if $u$ belongs to the range of $I_{A}^{0}$, then $u=I_{A}^{0}(b)$ for some $b\in C^{\infty}(M,\C^n)$.
By item (2) in Lemma  \ref{lemma:CR} we can find a smooth 1-form $\beta$ such that $\star d_{A}\beta=b$
and $d_{A}\star\beta=0$. By Theorem \ref{thm:surjectiveon1forms} we can find $w\in {\mathcal S}^{\infty}(\partial_{+}(SM),\C^n)$ such that $(I_{A}^{1})^*(w)=\beta$. Using again the first identity in Proposition \ref{prop:factorization} we see that $u=-2\pi P_{-}w$ as desired.

The proof of the second claim in Theorem \ref{thm:range_unitary} is similar.  Suppose that
$u=P_{+}w+I_{A}^{1}(\eta)$ with $\eta\in\mathfrak{H}_{A}$. Then the second identity in Proposition \ref{prop:factorization} shows that $u$ belongs to the range of $I^{1}_{A}$. 
Conversely,  if $u$ belongs to the range of $I_{A}^{1}$, then $u=I_{A}^{1}(\alpha)$ for some
$\alpha\in\Lambda^{1}(M)$. By  item (1) in Lemma  \ref{lemma:CR} we can find smooth 
functions $a$ and $p$ in $M$ and $\eta\in\mathfrak{H}_{A}$ such that
$\alpha=d_{A}p+\star d_{A}a+\eta$ with $p|_{\partial M}=0$.  By Theorem \ref{thm:surjective}
there is $w\in {\mathcal S}^{\infty}(\partial_{+}(SM),\C^n)$ such that $(I_{A}^{0})^*(w)=a$.
Since $I_{A}^{1}(d_{A}p)=0$ we see that $I_{A}^{1}(\alpha)=I_{A}^{1}(\eta)+I_{A}^{1}(\star d_{A}a)$.
Using the second identity in Proposition \ref{prop:factorization} we conclude that
$u=I_{A}^{1}(\alpha)=I_{A}^{1}(\eta)-2\pi P_{+}w$ and Theorem \ref{thm:range_unitary} is proved.
\end{proof}

It remains to prove Lemma \ref{lemma:CR}. The first part of the lemma is closely related to a Hodge decomposition. Consider the operator $d_A + d_A^*$ acting on sections of the bundle $(\Lambda M)^n$ of $\C^n$-valued differential forms. We will use Sobolev spaces with relative boundary condition corresponding to this first order operator, 
$$
H^s_R(M, (\Lambda M)^n) = \{ u \in H^s(M, (\Lambda M)^n) \,;\, tu = 0 \}
$$
where $s \geq 1$ and $t = j^*$ is the tangential trace with $j: \partial M \to M$ the natural inclusion.

\begin{Lemma} \label{lemma:hodgediracconnection}
If $(M,g)$ is a compact oriented manifold with boundary and if $s \geq 1$, the operator 
$$
d_A + d_A^*: H^s_R(M, (\Lambda M)^n) \to H^{s-1}(M,(\Lambda M)^n)
$$
has a finite dimensional kernel (independent of $s$) contained in $C^{\infty}(M,(\Lambda M)^n)$. The range of this operator is the intersection of $H^{s-1}(M,(\Lambda M)^n)$ and the $L^2$-ortho\-comple\-ment of its kernel.
\end{Lemma}
\begin{proof}
We first note that the boundary value problem 
$$
(d_A + d_A^*) u = f, \quad tu = g
$$
is elliptic. In fact, the ellipticity only depends on the principal symbols of the operator $d_A + d_A^*$ and the operator in the boundary condition \cite[Section 20.1]{Hor}. This reduces the statement to the ellipticity of the boundary value problem for the Hodge Dirac operator $d + d^*$ with relative boundary conditions, which is well known \cite[Exercise 5.11.9]{Ta}. It now follows from \cite[Theorem 20.1.2]{Hor} that for any $s \geq 1$ the map 
\begin{multline*}
H^s(M, (\Lambda M)^n) \to H^{s-1}(M,(\Lambda M)^n) \oplus H^{s-1/2}(\partial M, (\Lambda (\partial M))^n),  \\ u \mapsto ((d_A+d_A^*) u, tu)
\end{multline*}
is Fredholm and hence has finite dimensional kernel denoted by $K$. By \cite[Theorem 20.1.8]{Hor} the elements in $K$ are of class $C^{\infty}$.

If $u \in H^1_R(M, (\Lambda M)^n)$, we have 
\begin{align*}
(d_A + d_A^*)u = 0 &\Longleftrightarrow ((d_A + d_A^*)u,v)_{L^2} = 0 \text{ for all } v \in H^1_R(M, (\Lambda M)^n) \\
&\Longleftrightarrow (u,(d_A + d_A^*)v)_{L^2} = 0 \text{ for all } v \in H^1_R(M, (\Lambda M)^n).
\end{align*}
Thus $K$ is the $L^2$-orthocomplement of the range of $d_A + d_A^*$ on $H^1_R(M,(\Lambda M)^n)$, and since the range is closed the range is equal to $K^{\perp}$. If now $u$ is in the range of $d_A + d_A^*$ acting on $H^s_R(M,(\Lambda M)^n)$ then clearly $u \in H^{s-1}(M,(\Lambda M)^n) \cap K^{\perp}$. Conversely, if $u$ is in this set then $u = (d_A + d_A^*)w$ for some $w \in H^1_R(M,(\Lambda M)^n)$, and elliptic regularity \cite[Theorem 20.1.7]{Hor} implies that $w \in H^s_R(M,(\Lambda M)^n)$.
\end{proof}

Notice that $d_A + d_A^*$ maps even degree forms to odd degree forms and vice versa, and thus the kernel of this operator also splits into even and odd degrees. If $(M,g)$ is two-dimensional, the set of $1$-forms in the kernel of $d_A + d_A^*$ is precisely the set $\mathfrak{H}_A$ of $A$-harmonic $1$-forms considered earlier.

\begin{proof}[Proof of Lemma \ref{lemma:CR} item (1)]
Given a smooth $\C^n$ valued $1$-form $\alpha$, write $\alpha = \beta + \eta$ where $\eta$ is the $L^2$-projection of $\alpha$ to $\mathfrak{H}_A$. Then $\beta$ is orthogonal to the kernel of $d_A+d_A^*$, and by Lemma \ref{lemma:hodgediracconnection} there is $u \in H^1_R(M,(\Lambda M)^n)$ with $(d_A + d_A^*)u = \beta$. By elliptic regularity we have that $u$ is smooth, and writing $u$ in terms of its $0$-form and $2$-form parts, $u = u^0 + u^2$, we see that 
$$
d_A u^0 + d_A^* u^2 + \eta = \alpha.
$$
It is enough to take $p=u^0$ and $a = - \star u^2$, where the boundary condition $tu = 0$ reduces to $p|_{\partial M} = 0$.
\end{proof}


In what follows we will elaborate a bit more on the content of item (1) in Lemma \ref{lemma:CR}.
We shall see that the triviality of $\mathfrak{H}_{A}$ is equivalent to a solvability result for a suitable Cauchy-Riemann operator with  boundary conditions on a totally real subbundle with zero Maslov index. Observe that by Lemma \ref{lemma:hodgediracconnection} the claim $\mathfrak{H}_{A}=0$ is equivalent to the claim that the operator 
$$
D: C^{\infty}(M,\C^n) \oplus C^{\infty}(M,\C^n) \to C^{\infty}(M,(\Lambda^{1} M)^n), \ \ D(p,a) = d_A p + \star d_A a
$$
is surjective, when the boundary condition $p|_{\partial M}=0$ is imposed.

Recall that given a connection $A$ we may use the complex structure on $M$ (via the $\star$ operator) to introduce a $\bar{\partial}_{A}$ operator acting
on sections as $\bar{\partial}_{A}=(d_{A}-i\star d_{A})/2$.
Equivalently $\bar{\partial}_{A}=\bar{\partial}+A^{0,1}$, where
$A^{0,1}=(A-i\star A)/2$. We also have $\partial_{A}=(d_{A}+i\star d_{A})/2=\partial+A^{1,0}$, where
$A^{1,0}=(A+i\star A)/2$. Consider the unitary connection in $\C^n\oplus\C^n$ given by
\[\tilde{A}:=\left(
  \begin{array}{ c c }
     A & 0 \\
     0 & \bar{A}
  \end{array} \right).\]
Define $F\subset \C^n\oplus\C^n$ as the set of all $(z,w)\in  \C^n\oplus\C^n$ such that
$z+\bar{w}=0$. The subspace $F$ is totally real in the sense that $F\cap iF=\{0\}$.
Given $\alpha \in C^{\infty}(M,(\Lambda^{1}M)^n)$ consider the following boundary value problem for $u:M\to \C^{n}\oplus\C^n$:
\begin{equation}
\left\{\begin{array}{ll}
\bar{\partial}_{\tilde{A}}u=(\alpha^{0,1},\overline{\alpha^{1,0}}),\\
u(\partial M)\in F.\\
\end{array}\right.
\label{eq:CR}
\end{equation}
This problem fits the setting of the boundary value Riemann-Roch theorem, see Theorem C.1.10 in \cite{McS}. In our case the subspace $F$ has zero Maslov index  and hence the real Fredholm index
of $\bar{\partial}_{\tilde{A}}$ is $2n$.
Assume that $u$ is a solution to (\ref{eq:CR}) and let $u=(f,g)$. Then $\bar{\partial}_{A}f=\alpha^{0,1}$
and $\bar{\partial}_{\bar{A}}g=\overline{\alpha^{1,0}}$ with $(f+\bar{g})|_{\partial M}=0$.
Since $\bar{A}^{0,1}=\overline{A^{1,0}}$ we have $\partial_{A}\bar{g}=\alpha^{1,0}$.
If we now set $p:=(f+\bar{g})/2$ and $a:=i(\bar{g}-f)/2$, an easy calculation shows that
$d_{A}p+\star d_{A}a=\alpha$ and $p|_{\partial M}=0$.  It is easy to reverse the process: given
$p$ and $a$ with $d_{A}p+\star d_{A}a=\alpha$ and $p|_{\partial M}=0$ one checks that
$u=(p+ia,\bar{p}+i\bar{a})$ solves (\ref{eq:CR}). Hence the solvability of (\ref{eq:CR}) is equivalent to
the surjectivity of $D$ which in turn is equivalent to $\mathfrak{H}_{A}=0$.
The space $\mathfrak{H}_{A}$ then depends exclusively on the holomorphic data determined 
by the connection and the metric.

We now move to the proof of the second part of Lemma \ref{lemma:CR}. In this case one does not need to prescribe the boundary values of the $1$-form $\beta$, and it is enough to use a semiglobal solvability result. As above, consider the operator $d_A$, its adjoint $d_A^*$, and the corresponding Laplacian
$$
-\Delta_A = d_A^* d_A + d_A d_A^*.
$$
This operator acts on $\C^n$-valued graded forms, and it maps $k$-forms to $k$-forms. We will employ the following solvability result.

\begin{Lemma} \label{lemma:deltaA_surjective}
Let $(M,g)$ be a compact subdomain of a closed oriented surface $(S,g)$ such that global isothermal coordinates exist near $M$, and let $Q$ be a first order differential operator with smooth coefficients acting on $C^{\infty}(M, (\Lambda M)^n)$. Given any $f \in C^{\infty}(M, (\Lambda M)^n)$, the equation 
$$
(-\Delta_A + Q)u = f
$$
has a solution $u \in C^{\infty}(M, (\Lambda M)^n)$.
\end{Lemma}
\begin{proof}
Since $-\Delta_A + Q$ has diagonal principal part which is of real principal type, this should follow from semiglobal solvability results \cite[Section 26.1 and references therein]{Hor}. However, we can also do this directly. Note first that 
$$
-\Delta_A + Q = -\Delta I_{n \times n} + Q'
$$
where $\Delta$ is the usual Hodge Laplacian on graded differential forms and $Q'$ is another first order operator. We extend $Q'$ and $f$ smoothly to $S$. The assumption allows us to find global isothermal coordinates in an open set $U$ containing $M$ so that $g = c e$ in $U$ where $c$ is a positive function and $e$ is the Euclidean metric. We think of $U$ as a bounded open set in $\re^2$, and the Hodge Laplacian satisfies 
$$
-\Delta_g = c^{-1} \Delta_e + Q''
$$
for some first order operator $Q''$. To prove the lemma, it is enough to solve the equation 
$$
r_M (-\Delta_e I_{n \times n} + Q) u = c f
$$
for $u \in C^{\infty}(S, (\Lambda S)^n)$, where $Q$ is a first order operator and $r_M$ is the restriction to $M$.

Consider the operator 
$$
T = r_M (-\Delta_e I_{n \times n} + Q): C^{\infty}(S, (\Lambda S)^n) \to C^{\infty}(M, (\Lambda M)^n).
$$
This is a continuous linear operator with adjoint 
$$
T^* = (-\Delta_e I_{n \times n} + Q^*)|_{\mathcal{D}'_M}: \mathcal{D}'_M(S, (\Lambda S)^n) \to \mathcal{D}'(S, (\Lambda S)^n)
$$
where $\mathcal{D}'_M$ denotes distributions in $S$ with support in $M$. To show that $T$ is surjective, it is enough to check that $T^*$ is injective and has weak$*$ closed range \cite[Theorem 37.2]{Treves}. But if $T^* u = 0$, then $u \in C^{\infty}(S, (\Lambda S)^n)$ and $\text{supp}(u) \subset M$ by elliptic regularity. Now we can exploit a global Carleman estimate for $h > 0$ small, for instance 
\begin{equation} \label{hodge_carlemanestimate}
\norm{v}_{L^2(U)} \leq C h \norm{e^{\varphi/h} (-\Delta_e I_{n \times n} + Q^*) e^{-\varphi/h} v}_{L^2(U)}, \quad v \in C^{\infty}_c(U, (\Lambda U)^n),
\end{equation}
where $\varphi(x) = x_1 + h x_1^2/(2\varepsilon)$ for suitable fixed $\varepsilon > 0$. To obtain this estimate, we use the result of \cite[Lemma 2.1]{SaloTzou} which says that for $h \ll \varepsilon \ll 1$, 
$$
\norm{v}_{L^2(U)} + h \norm{\nabla v}_{L^2(U)} \leq C h \sqrt{\varepsilon} \norm{e^{\varphi/h} (-\Delta_e I_{n \times n}) e^{-\varphi/h} v}_{L^2(U)}, \quad v \in C^{\infty}_c(U, (\Lambda U)^n).
$$
(Note that the Hodge Laplacian with Euclidean metric is just the usual Laplacian acting on each component of the forms.) The estimate \eqref{hodge_carlemanestimate} is a consequence of the above inequality if $\varepsilon > 0$ is chosen small enough. Applying \eqref{hodge_carlemanestimate} to $v = e^{\varphi/h} u$ shows that $u = 0$, proving the injectivity of $T^*$. Since $T^*$ is the restriction of an elliptic pseudodifferential operator, the closed range condition also follows as in \cite[proof of Theorem 6.3.1]{DH2}. This shows surjectivity of $T$.
\end{proof}

The next result immediately implies the second part of Lemma \ref{lemma:CR}.

\begin{Lemma} \label{lemma:CR2}
Given $f \in C^{\infty}(M,\C^n)$ there is a smooth $\C^n$-valued 1-form $\beta$ with $d_{A}^* \beta = f$
and $d_{A} \beta = 0$.
\end{Lemma}
\begin{proof}
Look for $\beta$ of the form $\beta = d_A u^0 + d_A^* u^2$ for some smooth forms $u^0, u^2$. Then we need that 
\begin{align*}
d_A^* d_A u^0 + (F_A)^* u^2 &= f, \\
d_A d_A^* u^2 + F_A u^0 &= 0,
\end{align*}
where $F_A = d_A \circ d_A = dA + A \wedge A$ is the curvature of $d_A$. Writing $u = u^0 + u^2$, these equations are equivalent to
$$
-\Delta_A u + Qu = f
$$
for some operator $Q$ of order $0$. Lemma \ref{lemma:deltaA_surjective} implies the existence of a smooth solution $u$, which gives the required $\beta$.
\end{proof}

\end{document}